\documentclass[oneside, 10pt]{amsart}

\usepackage{amsmath, amsfonts, amssymb, amsthm, graphics, times}

\newtheorem{theorem}{Theorem}[section]
\newtheorem{lemma}[theorem]{Lemma}
\newtheorem{corollary}[theorem]{Corollary}
\newtheorem{proposition}[theorem]{Proposition}

\theoremstyle{definition}

\newtheorem{remark}[theorem]{Remark}

\numberwithin{equation}{section}

\begin{document}

\title[Differential-free     characterisation of smooth mappings]{Differential-free   characterisation of smooth
mappings with controlled growth}

\author{Marijan Markovi\'{c}}

\address{Faculty of Natural Sciences and Mathematics\endgraf
University of Montenegro\endgraf
Cetinjski put b.b.\endgraf
81000 Podgorica\endgraf
Montenegro}

\email{marijanmmarkovic@gmail.com}

\begin{abstract}
In   this             paper we  give some generalizations  and improvements of the  Pavlovi\'{c} result   on the
Holland-Walsh type characterization of the Bloch space     of  continuously differentiable (smooth) functions in
the unit ball in $\mathbf{R}^m$.
\end{abstract}



\maketitle

\section{Introduction and the main result}

We consider the space $\mathbf{R}^m$ equipped with the standard norm $|\zeta |$           and the scalar product
$\left<\zeta, \eta\right>$  for $\zeta\in \mathbf{R}^m$ and $\eta\in \mathbf{R}^m$.  We denote by $\mathbf{B}^m$
the    unit ball in $\mathbf{R}^m$. Let $\Omega\subseteq\mathbf{R}^m$ be a domain. For a differentiable  mapping
$f :\Omega\rightarrow \mathbf{R}^n$, denote by $Df(\zeta)$ its  differential      at  $\zeta \in \Omega$, and by
\begin{equation*}
\|Df(\zeta)\| \ = \sup _{\ell \in \partial \mathbf{B}^{m}} |Df(\zeta)\ell|
\end{equation*}
the norm of the linear  operator  $Df(\zeta) : \mathbf{R}^m\rightarrow \mathbf{R}^n$.

This paper is mainly  motivated by the      following surprising  result of    Pavlovi\'{c} \cite{PAVLOVIC.PEMS}.

\begin{proposition}[Cf. \cite{PAVLOVIC.PEMS}]\label{PR.PAVLOVIC}
A continuously differentiable  complex-valued  function  $f(\zeta)$ in  the unit ball $\mathbf{B}^m$  is a Bloch
function, i.e.,
\begin{equation*}
\sup_{\zeta\in \mathbf{B}^m}  (1-|\zeta|^2) \|Df (\zeta)\|
\end{equation*}
is finite, if and only if the following quantity if finite:
\begin{equation*}
\sup_{\zeta,\,  \eta\in \mathbf{B}^m,\, \zeta\ne \eta}\sqrt { 1-|\zeta|^2 }\sqrt{  1-|\eta|^2  }  \frac{|f(\zeta)
-f (\eta)|}{|\zeta-\eta|}.
\end{equation*}
Moreover, these  numbers are equal.
\end{proposition}

As Pavlovi\'{c} observed in \cite{PAVLOVIC.PEMS}, the above result is actually two-dimensional. Namely, if   one
proves it for continuously  differentiable functions $\mathbf{B}^2 \rightarrow\mathbf{C}$, then the general case
(the case of  continuously  differentiable  functions   $\mathbf{B}^m  \rightarrow \mathbf{C}$)  follows from it.
We  give  a  proof of Proposition \ref{PR.PAVLOVIC}  in the next section following our main result.

Since  for  an    analytic  function   $f(z)$ in the unit  disc $\mathbf{B}^2$  we  have   $\|Df(z) \|= |f'(z)|$
for every $z\in \mathbf{B}^2$, the first part of            Proposition \ref{PR.PAVLOVIC}  (without the equality
statement) is the Holland--Walsh  characterization of analytic  functions in the  Bloch space   in the unit disc.
See Theorem 3 in  \cite{HOLLAND.WALSH.MATH.ANN} which says that  $f(z)$  is a Bloch function      if and only if
\begin{equation*}
\sqrt {1-|z|^2 }\sqrt {1-|w|^2 } \frac{|f(z) -f(w)|}{|z-w|}
\end{equation*}
is bounded    as a  function of two variables  $z\in \mathbf{B}^2$ and         $w\in \mathbf{B}^2$ for  $z\ne w$.
This  characterisation of analytic  Bloch functions in the unit ball is given   by Ren and Tu in \cite{REN.PAMS}.

Our aim here  is to  obtain  a characterisation result   similar as in Proposition \ref{PR.PAVLOVIC}          of
continuously    differentiable mappings that  satisfy a certain  growth condition.  Before we formulate our main
theorem  we need to  introduce some notation.

Let $\mathbf{w}(\zeta)$ be an everywhere positive continuous function in a domain $\Omega\subseteq \mathbf{R}^m$
(a weight  function in $\Omega$).  We will consider continuously differentiable mappings in $\Omega$ that    map
this domain into $\mathbf{R}^n$                                     and  satisfy the following  growth condition
\begin{equation*}
\|f\|_ \mathbf{w} ^\mathbf{b} := \sup_{\zeta \in \Omega} \mathbf{w}(\zeta) \|D f(\zeta)\|<\infty.
\end{equation*}

We say that  $\|f\|^\mathbf{b}_\mathbf{w}$ is  the $\mathbf{w}$-Bloch  semi-norm  of the mapping $f$ (it is easy
to check  that it has indeed  all  semi-norm properties). We   denote by $\mathcal {B}_\mathbf{w}$ the  space of
all continuously differentiable mappings                  $f:\Omega\rightarrow \mathbf{R}^n$   with  the  finite
$\mathbf{w}$-Bloch semi-norm. The space $\mathcal {B}_\mathbf{w}$ we call $\mathbf{w}$-Bloch space. If  $\Omega=
\mathbf{B}^m$ and $\mathbf{w}(\zeta)= 1-|\zeta|^2$ for $\zeta\in \mathbf{B}^m$, we just say the Bloch space, and
denote it by $\mathcal{B}$.

In  the sequel  we  will consider the $\mathbf{w}$-distance between $\zeta\in \Omega$ and $\eta\in\Omega$, which
is  obtained in the following way:
\begin{equation*}
d_\mathbf{w} (\zeta,\eta)  =  \inf_\gamma \int_\gamma \frac {|d\omega|}{\mathbf{w} (\omega)},
\end{equation*}
where  the infimun is taken over all piecewise smooth curves $\gamma\subseteq \Omega$ connecting   $\zeta$   and
$\eta$. It is well known that $d_\mathbf{w} (\zeta,\eta)$ is a distance function in the domain  $\Omega$.

One  of our  aims in this paper is to     give a  differential-free  description of the $\mathbf{w}$-Bloch space
and a differential-free expression  for  $\mathbf{w}$-Bloch    semi-norm.   In order to do that,  for    a given
$\mathbf{w} (\zeta)$                 in a domain $\Omega$, we now  introduce a new  everywhere positive function
$\mathbf{W}(\zeta,\eta)$ on the    product domain $\Omega\times \Omega$ that  satisfies the       following four
conditions. For every $\zeta\in \Omega$  and $\eta\in \Omega$,
\begin{equation*}\begin{split}
&(W_1)\quad  \mathbf{W}(\zeta,\eta)  =  \mathbf{W}(\eta,\zeta);
\\&(W_2)\quad  \mathbf{W}(\zeta,\zeta) =  \mathbf{w}  (\zeta);
\\&(W_3)\quad  \liminf_{\eta\rightarrow \zeta} \mathbf{W}(\zeta,\eta)\ge \mathbf{W}(\zeta,\zeta)
= \mathbf{w} (\zeta);
\\& (W_4)\quad  d_\mathbf{w}(\zeta,\eta)\mathbf{W}(\zeta,\eta) \le |\zeta-\eta|.
\end{split}\end{equation*}
We say that $\mathbf{W}(\zeta,\eta)$ is admissible  for $\mathbf{w} (\zeta)$.

Of course,  one can pose  the existence  question concerning  $\mathbf{W}(\zeta,\eta)$   if $\mathbf{w} (\zeta)$
is given.   In the sequel we will prove that the following functions $\mathbf{W}(\zeta,\eta)$ are admissible for
the  given functions   $\mathbf{w} (\zeta)$.
\begin{enumerate}
\item The function  \begin{equation*}
\mathbf{W}(\zeta,\eta) =
\left\{
\begin{array}{ll}
\mathbf{w}(\zeta) , & \hbox{if $\zeta=\eta$,} \\
{|\zeta - \eta|}/{  d_\mathbf{w} (\zeta,\eta)}, & \hbox{if $\zeta\ne \eta$}.
 \end{array}
\right.
\end{equation*}
in $\Omega\times \Omega$ is  admissible for  any given $\mathbf{w}(\zeta)$ in $\Omega$.
\item If $\mathbf{w} (\zeta)   = 1-|\zeta|^2$ for $\zeta\in\mathbf{B}^m$,  then    $d_\mathbf{w}(\zeta,\eta)$ is
the  hyperbolic distance  in the unit  ball $\mathbf{B}^m$. One of the   admissible  functions is
\begin{equation*}
\mathbf{W}(\zeta,\eta) = \sqrt{1-|\zeta|^2} \sqrt{1-|\eta|^2}.
\end{equation*}
This is shown  in  the next  section. From this fact we deduce  the Pavlovi\'{c}   result    stated in the above
proposition.
\item If  $\Omega$ is a  convex domain  and if $\mathbf {w}(\zeta)$ is  a decreasing function in $|\zeta|$, then
\begin{equation*}
\mathbf{W}(\zeta,\eta) =\min \{\mathbf{w}(\zeta),\mathbf{w}(\eta)\}
\end{equation*}
is admissible for  $\mathbf{w} (\zeta)$. It would be of interest to find such  simple admissible  functions  for
more  general domains  $\Omega$ and/or more general functions   $\mathbf{w}$.
\end{enumerate}

For      a                 mapping    $f:\Omega\rightarrow \mathbf{R}^n$       introduce    now  the    quantity
\begin{equation*}
\|f\|_\mathbf{W} ^\mathbf{l}\ \ \ : = \sup _{\zeta,\,  \eta\in \Omega,\,  \zeta\ne \eta} \mathbf{W}(\zeta,\eta)
\frac {|f(\zeta)-f(\eta)|}{|\zeta-\eta|}.
\end{equation*}
We call it the $\mathbf{W}$-Lipschitz semi-norm (it is also an easy task to check that it is indeed a semi-norm).
The space of all continuously differentiable mappings  $f:\Omega\rightarrow \mathbf{R}^n$ for which          its
$\mathbf{W}$-Lipschitz semi-norm   $\|f \|_{\mathbf{W}}^\mathbf{l}$ is finite  is  denoted                    by
$\mathcal {L}_\mathbf{W}$. Note that  if   $\mathbf{W}(\zeta,\eta)$  is not symmetric,      we can replace it by
$\tilde {\mathbf{W}}(\zeta,\eta)=\max \{\mathbf{W}(\zeta,\eta),\mathbf{W}(\eta,\zeta)\}$ which produces the same
Lipschitz type   semi-norm.

Our  main  result  in this paper   shows that for any continuously differentiable mapping   $f:\Omega\rightarrow
\mathbf{R}^n$  we have $\|f \| _\mathbf{w}^\mathbf{b}  = \|f\|_ \mathbf{W}^\mathbf{l}$;  i.e.,               the
$\mathbf{w}$-Bloch semi-norm is equal to the $\mathbf{W}$-Lipschitz semi-norm of the mapping  $f$.          As a
consequence we have  the coincidence of the two spaces  $\mathcal {B} _ \mathbf{w}  =  \mathcal {L} _\mathbf{W}$.
Thus,  the space $\mathcal {B}_\mathbf{w}$ may be described as
\begin{equation*}
\mathcal {B}_\mathbf{w} =\left  \{f: \Omega\rightarrow \mathbf{R}^n:
\sup _{\zeta,\,  \eta\in \Omega,\,  \zeta\ne \eta}
\mathbf{W}(\zeta,\eta) \frac  {|f(\zeta)-f(\eta)|}{|\zeta-\eta|}< \infty\right\},
\end{equation*}
where  $\mathbf{W}(\zeta,\eta)$ is any admissible function for $\mathbf{w}(\zeta)$.   This is the content of the
following theorem.

\begin{theorem}\label{TH.MAIN}
Let $\Omega\subseteq\mathbf{R}^m$ be a domain  and let $f:\Omega\rightarrow \mathbf{R}^n$      be a continuously
differentiable mapping.                 Let $\mathbf{w}(\zeta)$ be positive and continuous in $\Omega$,  and let
$\mathbf{W}(\zeta,\eta)$ be  an  admissible function for $\mathbf{w}(\zeta)$.              If one of the numbers
$\|f\|_\mathbf{w}^\mathbf{b}$ and $\|f\|_\mathbf{W}^\mathbf{l}$ is finite,  then    both  numbers are finite and
equal.
\end{theorem}

\begin{proof}
For one direction,    assume that $\mathbf{W}$-Lipschitz   semi-norm of the mapping  $f$ is finite,   i.e., that
the  quantity
\begin{equation*}
\sup _{\zeta,\, \eta\in \Omega,\, \zeta\ne \eta} \mathbf{W}(\zeta,\eta)\frac {|f(\zeta)-f(\eta)|}{|\zeta-\eta|}
\end{equation*}
is finite. We will show that  $\|f\|_\mathbf{w} ^\mathbf{b}\le \|f\|_\mathbf{W}^\mathbf{l}$, which  implies that
$\|f\|_\mathbf{w} ^\mathbf{b}$ is also  finite.

If we have  in mind that
\begin{equation*}
\limsup_{\omega\rightarrow  \zeta } \frac{|f(\zeta) - f(\omega)|}{|\zeta - \omega|} = \|Df(\zeta)\|
\end{equation*}
for every $\zeta\in \Omega$,  we obtain
\begin{equation*}\begin{split}
\|f\|_\mathbf{W}^\mathbf{l}      \ \  &
  =  \sup _{\eta,\, \omega\in \Omega,\,  \eta\ne \omega}   \mathbf{W}(\eta,\omega)
\frac {|f(\eta)-f(\omega)|}{|\eta - \omega|}\ge \limsup_{\omega \rightarrow \zeta} \mathbf{W}(\zeta,\omega)\frac
{|f(\zeta)-f(\omega)|}{|\zeta-\omega|}
\\& \ge \liminf_{\omega\rightarrow \zeta} \mathbf{W}(\zeta,\omega)
\limsup_{\omega\rightarrow \zeta}\frac {|f(\zeta)-f(\omega)|}{|\zeta-\omega|}
=  \mathbf{W}(\zeta,\zeta) \|Df(\zeta)\|\\& = \mathbf{w}(\zeta)\|Df(\zeta)\|.
\end{split}\end{equation*}
It follows that
\begin{equation*}
\|f\|_\mathbf{W}^\mathbf{l} \ge\sup_{\zeta\in \Omega}\mathbf{w}(\zeta)\|Df(\zeta)\|= \|f\|_\mathbf{w}^\mathbf{b},
\end{equation*}
which we aimed to prove.

Assume now that $\|f\|_\mathbf{w}^\mathbf{b}$ is finite.                 We  will  prove  the reverse inequality
$\|f\|_\mathbf{W}^\mathbf{l} \le\|f\|_\mathbf{w}^\mathbf{b}$.      Let $\zeta\in\Omega$ and $\eta\in \Omega$  be
arbitrary  and   different  and let  $\gamma\subseteq \Omega$  be any piecewise smooth curve parameterized by $t
\in [0,1]$ that   connects $\zeta$ and $\eta$, i.e., for which  $\gamma(0)=\zeta$  and  $\gamma(1)=\eta$.  Since
$\|f\|_\mathbf{w}^\mathbf{b}$  is finite,  we obtain
\begin{equation*}\begin{split}
|f(\zeta) - f(\eta)| &= | (f\circ \gamma) (1) - (f\circ \gamma) (0)  |
= \left|\int_0^1 ((f\circ\gamma)(t) )'dt\right|\\&
= \left|\int_0^1  Df ( \gamma (t) ) \gamma'(t) dt\right|
\le\int_0^1 \left| Df ( \gamma (t) ) \gamma'(t) \right|dt
\\& \le \int_0^1 \|Df(\gamma (t))\| |\gamma'(t)|dt
\le\|f\|_\mathbf{w}^\mathbf{b} \int_0^1  \frac {|\gamma'(t)|dt}{\mathbf{w} (\gamma(t))}
\\&=\|f\|_\mathbf{w}^\mathbf{b} \int_\gamma \frac {|d\omega|}{\mathbf{w} (\omega)}.
\end{split}\end{equation*}
If we take infimum over all  such    curves  $\gamma$   we obtain
\begin{equation*}\begin{split}
|f(\zeta) - f(\eta )|\le \|f\|_\mathbf{w}^\mathbf{b} d_{\mathbf{w}} (\zeta,\eta).
\end{split}\end{equation*}
Because of our conditions    posed on the function                           $\mathbf{W}(\zeta,\eta)$,   we have
\begin{equation*}
\mathbf{W}(\zeta,\eta)\frac {|f(\zeta) -  f(\eta)|}{|\zeta-\eta|}   \le
\mathbf{W}(\zeta,\eta)\frac {d_\mathbf{w}(\zeta,\eta)}{|\zeta-\eta|} \|f\|_\mathbf{w} ^ \mathbf{b} \le
\|f\|_\mathbf{w}^\mathbf{b}.
\end{equation*}
Therefore,
\begin{equation*}
\|f\|_{\mathbf{W}}^ \mathbf{l} \ \ \ =\sup_{\zeta,\, \eta\in \Omega,\,  \zeta\ne\eta} \mathbf{W}(\zeta,\eta)
\frac {|f(\zeta) -  f(\eta)|}{|\zeta-\eta|} \le \|f\|_\mathbf{w}^\mathbf{b},
\end{equation*}
which we wanted to prove.
\end{proof}

\begin{remark}
Let $\mathbf{w}(\zeta)$ be a weight in a domain $\Omega\subseteq \mathbf{R}^m$.         Observe   that   we have
\begin{equation*}
\sup_{\zeta\in \Omega} \mathbf{w}(\zeta)\ \  \ = \sup_{\zeta,\,  \eta\in \Omega,\, \zeta\ne\eta}
\mathbf{W}(\zeta,\eta),
\end{equation*}
where $\mathbf{W}(\zeta,\eta)$ is admissible for $\mathbf{w}(\zeta)$. This remark    is a  direct consequence of
the fact that   we can set the identity  $f(\zeta) = \mathrm {Id}(\zeta)$ in Theorem \ref{TH.MAIN}.
\end{remark}

\section{On the Pavlovi\'{c} result}

As we have already  said,       if we take $\mathbf{w} (\zeta)= 1-|\zeta|^2$ for $\zeta\in \mathbf{B}^m$,   then
$\mathbf{w}$-distance is the hyperbolic distance. For the hyperbolic distance between    $\zeta\in \mathbf{B}^m$
and $\eta\in \mathbf{B}^m$  we will use the usual  notation $\rho(\zeta,\eta)$.

It is well known that the hyperbolic  distance  is invariant  under M\"{o}bius transforms of the unit ball; i.e.,
if $T: \mathbf{B}^m \rightarrow \mathbf{B}^m$ is  a  M\"{o}bius   transform,  then we  have
\begin{equation*}
\rho ( T (\zeta),T(\eta)) = \rho (\zeta,\eta)
\end{equation*}
for every $\zeta\in \mathbf{B}^m$ and $\eta\in \mathbf{B}^m$.

Up to an   orthogonal  transform, a   M\"{o}bius  transform of  the unit ball $\mathbf {B}^m$ onto itself can be
represented as
\begin{equation*}
T_\zeta  (\eta ) = \frac { - (1-|\zeta|^2)(\zeta- \eta) - |\zeta- \eta|^2 \zeta}{[\zeta,\eta]^2},  \quad \eta\in
\mathbf{B}^m
\end{equation*}
for $\zeta\in \mathbf{B}^m$, where
\begin{equation*}
[\zeta, \eta]^2 = 1-2\left<\zeta,\eta\right> + |\zeta|^2|\eta|^2.
\end{equation*}

It is known that
\begin{equation*}
|T_\zeta \eta  |  =\frac {|\zeta-\eta|}{[\zeta,\eta]}\quad \text{and}\quad
1 - |T_\zeta \eta  |^2  =\frac {(1-|\zeta|^2) (1-|\eta|^2)}{[\zeta,\eta]^2}
\end{equation*}
for every $\zeta\in \mathbf{B}^m$ and $\eta\in\mathbf{B}^m$.

Particularly,  one easily  calculates
\begin{equation*}
\rho(0,\omega) = \frac 12\log \frac{1+ |\omega|}{1 - |\omega|}
\end{equation*}
for $\omega\in\mathbf{B}^m$. Because of the invariance with respect to the group of M\"{o}bius transforms of the
unit    ball, the hyperbolic distance between  $\zeta\in \mathbf{B}^m$ and $\eta\in\mathbf{B}^m$          can be
expressed as
\begin{equation*}
\rho(\zeta,\eta)=  \frac 12\log \frac{1+ |T_\zeta \eta|}{1- |T_\zeta \eta|} =   \mathrm{atanh}\, |T_\zeta(\eta)|.
\end{equation*}

For all mentioned facts and identities above we refer the reader to Ahlfors \cite{AHLFORS.BOOK} or      Vuorinen
\cite{VUORINEN.BOOK}.

Proposition    \ref{PR.PAVLOVIC}    can be seen as a  consequence of our main result and the following elementary
lemma which proves that $\mathbf{W}(\zeta,\eta)= \sqrt {1-|\zeta|^2} \sqrt {1-|\eta|^2}$ has  $W_4$-property, and
therefore     it  is  admissible  for $\mathbf{w}(\zeta) = 1-|\zeta|^2$.

\begin{lemma}
The function  $\mathbf{W}(\zeta,\eta)= \sqrt {1-|\zeta|^2} \sqrt {1-|\eta|^2}$         satisfies  the  inequality
\begin{equation*}
\rho(\zeta,\eta)  \mathbf{W}(\zeta,\eta)\le |\zeta - \eta|
\end{equation*}
for every  $\zeta\in \mathbf{B}^m$ and $\eta\in\mathbf{B}^m $.
\end{lemma}

\begin{proof}
We  will   first establish  the following special case of the inequality we need:
\begin{equation*}
\rho(0,\omega){\sqrt {1-|\omega|^2}} \le   |\omega|
\end{equation*}
for $\omega\in\mathbf{B}^m$.

Since
\begin{equation*}
\rho(0,\omega) = \frac 12\log \frac{1+ |\omega|}{1 - |\omega|},
\end{equation*}
if we take $t =|\omega|$, the above  inequality is equivalent to the  following  one:
\begin{equation*}
\frac 12 \log\frac{1+t}{1-t} \le \frac {t}{ \sqrt {1-t^2}  },
\end{equation*}
where  $0\le t<1$.       Denote the difference of the  left-hand side minus the right-hand side by $F(t)$.   Then
we have
\begin{equation*}
F'(t) = -\frac {1}{(1 - t^2)^{3/2}} + \frac 1{1 - t^2}, \quad 0<t<1.
\end{equation*}
Since $F'(t)< 0$ for $0<t<1$,  it follows that $F(t)$ is a decreasing function in $[0,1)$.              Therefore,
$F(t)\le F(0) = 0$, which implies the inequality  we aimed to prove.

It the inequality we have just proved, let us take $\omega=T_\zeta \eta$, where $\zeta\in\mathbf{B}^m$ and  $\eta
\in \mathbf{B}^m$  are arbitrary.  Then we have
\begin{equation*}
\rho(0,\omega) = \rho(T_\zeta \zeta, T_\zeta\eta ) = \rho(\zeta,\eta),
\end{equation*}
\begin{equation*}
\sqrt{1-|\omega|^2} = \sqrt{1-|T_\zeta \eta |^2} = \frac {\sqrt{1-|\zeta|^2} \sqrt{1-|\eta|^2}}{[\zeta,\eta]},
\end{equation*}
as well as
\begin{equation*}
|\omega| = |T_\zeta \eta | = \frac{|\zeta-\eta|}{[\zeta,\eta]}.
\end{equation*}
If we substitute all above expressions,                   we obtain the inequality in the statement  of our lemma.
\end{proof}

\begin{remark}
One more   expression for the hyperbolic distance in the unit ball is given by
\begin{equation*}
\sinh^2 \rho(\zeta,\eta)  = \frac{|\zeta - \eta |^2}{(1-|\zeta|^2)(1-|\eta|^2)}
\end{equation*}
(see \cite{VUORINEN.BOOK}).  Using the elementary inequality $t\le  \sinh   t$, as suggested by the referee,  one
deduces  the inequality in the above lemma.
\end{remark}

\section{Some other consequences of the main theorem}

In this section we will derive some new consequences of our main result.

\begin{corollary}\label{CORO.MIN}
Let $\mathbf{w}(\zeta)$ be an everywhere positive, continuous and decreasing function of $|\zeta|$    in a convex
domain  $\Omega\subseteq \mathbf{R}^m$. Then we have
\begin{equation*}
\sup_{\zeta \in \Omega} \mathbf{w}(\zeta)\|Df(\zeta)\| \ \ \ =
\sup_{\zeta,\, \eta\in \Omega,\, \zeta\ne\eta}  \min \{\mathbf{w} (\zeta),\mathbf{w}(\eta)\}
\frac {|f(\zeta) - f(\eta)|}{|\zeta - \eta|}
\end{equation*}
for every continuously differentiable mapping $f:\Omega\rightarrow \mathbf{R}^n$.
\end{corollary}

\begin{proof}
Let
\begin{equation*}
\mathbf{W}(\zeta,\eta)   = \min \{\mathbf{w} (\zeta), \mathbf{w} (\eta)\},
\end{equation*}
for $(\zeta,\eta)\in\Omega\times \Omega$. We have only  to check if $\mathbf{W}(\zeta,\eta)$ satisfies conditions
$(W_1) -   (W_4)$ and to  apply our  main theorem.

It is clear that $\mathbf{W}(\zeta,\eta)$   is symmetric, and that $\mathbf{W}(\zeta,\zeta) = \mathbf{w} (\zeta)$.
Since $\mathbf{W}(\zeta,\eta)$ is continuous  in                 $\Omega\times \Omega$, the $(W_3)$-condition for
$\mathbf{W}(\zeta,\eta)$ obviously  holds.   Therefore, it remains  to check if the following  inequality is true:
\begin{equation*}
d_\mathbf{w}(\zeta,\eta)\min \{\mathbf{w} (\zeta), \mathbf{w} (\eta)\}\le |\zeta-\eta|
\end{equation*}
for every $(\zeta,\eta)\in  \Omega\times \Omega$.

Let  $\zeta\in \Omega$  and  $\eta\in \Omega$ be arbitrary and fixed   and let  $\gamma\subseteq \Omega$ be among
piecewise smooth  curves  that  joint   $\zeta$ and $\eta$.  We have
\begin{equation*}\begin{split}
d_\mathbf{w} (\zeta,\eta) & = \inf _\gamma \int_{\gamma} \frac {|d\omega |} {\mathbf{w}(\omega)}
\le\int_{[\zeta,\eta]}\frac {|d\omega|} {\mathbf{w}(\omega)}
\le \int_{[\zeta,\eta]} \max_{\omega \in [\zeta,\eta]}\left\{ \frac {1} {\mathbf{w}(\omega)} \right\} {|d\omega|}
\\&\le\max  \left\{\frac 1{\mathbf{w}(\zeta)},\frac 1{\mathbf{w}(\eta)}  \right\} \int_{[\zeta,\eta]} {|d\omega|}
= \max \left\{\frac 1{\mathbf{w}(\zeta)},\frac 1{\mathbf{w}(\eta)} \right\} |\zeta - \eta|
\\&= \min \{\mathbf{w} (\zeta), \mathbf{w} (\eta)\}^{-1}|\zeta - \eta|,
\end{split}\end{equation*}
where we     have used in the fourth step our assumption that  $\mathbf{w}(\omega)$ is decreasing in  $|\omega|$
and that    the maximum modulus of points on a line segment is attained at an endpoint. The   inequality we need
follows.
\end{proof}

\begin{remark}
Since the function $\mathbf{w} (\zeta)=1-|\zeta|^2$ is decreasing in $|\zeta|$ in the unit ball   $\mathbf{B}^m$,
the  above corollary produces a  new Holland-Walsh   type  characterisation of continuously differentiable Bloch
mappings. Notice that $\min\{A,B\}\le  \sqrt{A}\sqrt{B}$ for all non-negative numbers $A$ and $B$.    Because of
this inequality, it seems that Corollary \ref{CORO.MIN} improves the Pavlovi\'{c} result stated at the beginning
of the paper as  Proposition  \ref{PR.PAVLOVIC}.
\end{remark}

\begin{corollary}\label{CORO.D.RHO}
Let $\mathbf{w}(\zeta) $        be an everywhere positive and continuous  function  in a domain $\Omega$ and let
$d_\mathbf{w}(\zeta,\eta)$ be the $\mathbf{w}$-distance in  $\Omega$.  Then we have
\begin{equation*}
\sup_{\zeta\in \Omega} \mathbf{w}(\zeta) \|Df(\zeta)\|  \ \ \ = \sup_{\zeta,\, \eta\in\Omega,\,   \zeta\ne \eta}
\frac{|f(\zeta) - f(\eta)|}{d_\mathbf{w}(\zeta,\eta)}
\end{equation*}
for any continuously  differentiable  mappings  $f:\Omega\rightarrow \mathbf{R}^n$.
\end{corollary}

\begin{proof}
For $\zeta\in \Omega$ and $\eta\in \Omega$ let
\begin{equation*}
\mathbf{W}(\zeta,\eta) =
\left\{
\begin{array}{ll}
\mathbf{w}(\zeta) , & \hbox{if $\zeta=\eta$,} \\
{|\zeta - \eta|}/{  d_\mathbf{w} (\zeta,\eta)}, & \hbox{if $\zeta\ne \eta$}.
 \end{array}
\right.
\end{equation*}
It  is enough to show that $\mathbf{W}(\zeta,\eta)$ is admissible for $\mathbf{w}(\zeta)$.      It is clear that
$\mathbf{W}(\zeta,\eta)$  is        symmetric. The $(W_4)$-condition for  $\mathbf{W}(\zeta,\eta)$  is obviously
satisfied, and here it is optimal in some sense. Therefore, we have  only to  check if  $\mathbf{W}(\zeta,\eta)$
satisfies the $(W_3)$-condition:
\begin{equation*}
\liminf _{\eta\rightarrow \zeta} \mathbf{W}(\zeta,\eta)\ge\mathbf{W} (\zeta,\zeta).
\end{equation*}
This means that we need  to show that
\begin{equation*}
\liminf _{\eta\rightarrow \zeta}\frac{|\zeta - \eta|}{  d_\mathbf{w}(\zeta,\eta)} \ge \mathbf{w}(\zeta).
\end{equation*}
If we invert both sides, we obtain  that we have to prove
\begin{equation*}
\limsup _{\eta\rightarrow \zeta}  \frac {d_{\mathbf{w} }(\zeta,\eta)}{ |\zeta -\eta |}  \le
  \frac 1 {\mathbf{w}(\zeta)}.
\end{equation*}
for every $\zeta\in\Omega$.

Since this is a local question,  we may assume that $\eta $ is in a convex neighborhood of $\zeta$. Let $\gamma$
be among piecewise smooth curves in $\Omega$ connecting   $\zeta$ and $\eta$. We have
\begin{equation*}\begin{split}
\limsup_{\eta\rightarrow\zeta}\frac 1{|\zeta-\eta|}\inf_{\gamma}\int_\gamma\frac {|d\omega|}{\mathbf{w}(\omega)}
&\le\limsup_{\eta\rightarrow\zeta}\frac 1{|\zeta-\eta|}\int_{[\zeta,\eta]} \frac {|d\omega|}{\mathbf{w}(\omega)}
\\&= \lim_{\eta\rightarrow \zeta} \frac 1{|\zeta-\eta|} \int_{[\zeta,\eta]}\frac {|d\omega|}{\mathbf{w}(\omega)}
= \frac 1{\mathbf{w} (\zeta)},
\end{split}\end{equation*}
which we wanted to prove. The  equalities above follow because of continuity of the function $\mathbf{w}(\zeta)$.
\end{proof}

\begin{remark}
In the case $\mathbf{w}(\zeta)=(1-|\zeta|^2)^\alpha$ for $\zeta\in\mathbf{B}^2$, where  $\alpha>0$ is a constant,
Corollary  \ref{CORO.D.RHO}     is proved by Zhu  in \cite{ZHU.RMJM} for  analytic  functions (see Theorem 19 in
\cite{ZHU.RMJM}).      A variant of this  corollary is obtain in \cite{ZHU.JLMS}  (see also  Theorem 1 there for
 analytic functions).
\end{remark}

As a   special   case of the above corollary, we have the following one (certainly very well known for  analytic
Bloch  functions in the unit disc).

\begin{corollary}
A continuously differentiable  mapping  $f:\mathbf{B}^m\rightarrow \mathbf{R}^n$ is a Bloch mapping (i.e., $f\in
\mathcal{B}$) if and only if it is a Lipschitz mapping with respect to the  Euclidean and hyperbolic distance in
$\mathbf{R}^n$  and $\mathbf{B}^m$.  In other  words,   for the mapping  $f$,    there holds
\begin{equation*}
|f(\zeta) - f(\eta)| \le C \rho(\zeta,\eta)
\end{equation*}
for a constant $C$, if and only if $f\in \mathcal{B}$.                  Moreover,   the optimal  constant $C$ is
\begin{equation*}
C=\sup\{(1-|\zeta|^2)\| Df(\zeta)\|: \zeta\in \mathbf{B}^m\}
\end{equation*}
(for a given $f\in \mathcal{B}$)
\end{corollary}

\begin{remark}
The result of the last   corollary is  proved for  harmonic mappings of the unit disc into itself by  Colonna in
\cite{COLONNA.IUMJ}, where it is also found  that the constant $C$ is always  less or equal to $ 4/\pi$ for such
type of mappings.
\end{remark}

\subsection*{Acknowledgments}      I am thankful to the  referee for providing constructive comments and help in
improving the quality of this paper.


\begin{thebibliography}{10}

\bibitem{AHLFORS.BOOK}
L.V. Ahlfors,  \textit{M\"{o}bius transformations in several dimensions}, University of Minnesota,     School of
Mathematics, Minneapolis, Minnessota, Ordway Professorship Lectures in Mathematics (1981)

\bibitem{COLONNA.IUMJ}
F. Colonna, \textit{The Bloch constant of bounded harmonic mappings},  Indiana Univ. Math. J. \textbf{38} (1989),
829--840

\bibitem{HOLLAND.WALSH.MATH.ANN}
F. Holland and D. Walsh, \textit{Criteria for membership of Bloch space and its subspace, BMOA},       Math. Ann.
\textbf{273} (1986), 317--335

\bibitem{PAVLOVIC.PEMS}
M. Pavlovi\'{c},  \textit{On the Holland--Walsh characterization of Bloch functions},  Proc. Edinburgh Math. Soc.
\textbf{51} (2008), 439--441

\bibitem{REN.PAMS}
G. Ren and C. Tu, \textit{Bloch space in the unit ball of $\mathbb{C}^{n}$},  Proc. Amer. Math. Soc. \textbf{133}
(2005), 719--726

\bibitem{ZHU.JLMS}
K. Zhu, \textit{Distances and Banach spaces of holomorphic functions on complex domains},     J. London Math. Soc.
\textbf{49} (1994), 163--182

\bibitem{VUORINEN.BOOK}
M. Vuorinen, \textit{Conformal geometry and quasiregular mappings}, Springer--Verlag,                Berlin, 1988

\bibitem{ZHU.RMJM}
K. Zhu, \textit{Bloch Type Spaces of Analytic Functions}, Rocky Mountain J. Math. \textbf {23} (1993), 1143--1177

\end{thebibliography}
\end{document}